\documentclass[12pt,reqno]{amsart}
\usepackage{amssymb}
\usepackage{amsmath, mathtools}

\usepackage{amsthm}

\usepackage{amscd}

\newcommand{\RNum}[1]{\uppercase\expandafter{\romannumeral #1\relax}}

\usepackage{caption}

\usepackage[T2A]{fontenc}
\usepackage[utf8]{inputenc}
\usepackage[english]{babel}

\input{int.def} 

\usepackage[sort]{cite}
\usepackage{tikz-cd}
\usetikzlibrary{cd}
\usepackage{dirtytalk}
\usepackage[linktoc=page, colorlinks, linkcolor=blue, citecolor=blue]{hyperref}

\usepackage{xcolor}
\usepackage{centernot}

\usepackage{enumitem}

\usepackage{pgfplots}
\usepackage{multicol}

\pgfplotsset{compat=1.17}

\makeatletter
\renewenvironment{proof}[1][\proofname]{%
  \par\vspace{\topsep}%
  \normalfont\topsep6\p@\@plus6\p@\relax
  \trivlist
  \item[\hskip\labelsep\itshape #1\@addpunct{.}]\ignorespaces
}{%
  \endtrivlist
}
\makeatother

\numberwithin{equation}{section}

\DeclarePairedDelimiterX \ip[2]{\langle}{\rangle}{#1,#2}
\DeclarePairedDelimiterXPP \Prob[1]{\mathbb{P}}\{\}{}{ #1} 
\DeclarePairedDelimiterXPP \Probevent[1]{\mathbb{P}}(){}{#1} 

\usepackage[mathcal]{euscript}

\usepackage{titlesec}
\titleformat{\section}[runin]{\bfseries}{\thesection.}{3pt}{}[.]

\usepackage{geometry}
\newgeometry{vmargin={25mm}, hmargin={22mm,22mm}, footskip=10mm}   

\begin{document}
\title[TV between Bernoulli products, up to constants]{TV between Bernoulli products, up to constants}

\author{Gleb Smirnov}
\address{
Mathematical Sciences Institute, 
Australian National University, Canberra, Australia}
\email{gleb.smirnov@anu.edu.au}


\begin{abstract}
We obtain upper and lower bounds on the total variation distance between two product Bernoulli measures. The bounds are efficiently
computable and match up to an absolute constant factor.
\end{abstract}

\maketitle
\setcounter{section}{0}

\section{Main result}\label{main_results}

For probability measures \(\mu,\nu\) on a finite set, write:
\[
d_{\mathrm{TV}}(\mu,\nu)
=
\frac12\sum_x|\mu(x)-\nu(x)|.
\]
For \(p=(p_1,\ldots,p_n)\in[0,1]^n\), let
\[
\mathrm{Ber}(p)
=
\bigotimes_{i=1}^n\mathrm{Ber}(p_i),
\]
where \(\mathrm{Ber}(p_i)\) is the Bernoulli law on \(\{0,1\}\)
with parameter \(p_i\).
\smallskip%

Let $p,q\in[0,1]^n$. For $i=1,\ldots,n$, set:
\[
\lambda_i
:=
p_i(1-q_i)+(1-p_i)q_i.
\]
$\lambda_i$ is the probability 
that two independent
$\mathrm{Ber}(p_i)$ and $\mathrm{Ber}(q_i)$ variables disagree.
\smallskip%

Let
\[
D_i\sim\mathrm{Ber}(\lambda_i)
\quad\text{independently},
\]
and let
\[
a_i
:=
\begin{cases}
\dfrac{|p_i-q_i|}{\lambda_i},&\lambda_i>0,\\
0,&\lambda_i=0.
\end{cases}
\]
Set:
\[
G = \sum_{i=1}^n a_i^2 D_i.
\]
Our result is the following statement.

\begin{theorem}\label{main}
For each $n\ge1$ and $p,q\in[0,1]^n$,
\[
c\,\mathbb E\min\{1,\sqrt G\}
\le
d_{\mathrm{TV}}\bigl(\mathrm{Ber}(p),\mathrm{Ber}(q)\bigr)
\le
\mathbb E\min\{1,\sqrt G\},
\]
where \(c > 0\) is an absolute constant.
\end{theorem}
The constant \(1\) in the upper bound is sharp: for \(n=1\), both sides
equal \(|p_1-q_1|\).
\smallskip%

We write \(X\lesssim Y\) if \(X\le CY\) for an absolute constant \(C\), and \(X\asymp Y\) if both \(X\lesssim Y\) and \(Y\lesssim X\).
\smallskip%

Theorem~\ref{main} also gives a more explicit formula. Let
\begin{equation}\label{levy}
w(t)=\frac1{2\sqrt\pi}\,t^{-3/2}e^{-1/(4t)},\qquad t>0.
\end{equation}
This is a probability density on \((0,\infty)\), and
\begin{equation}\label{levy_laplace}
\int_0^\infty w(t)\,e^{-tu}\,dt=e^{-\sqrt u},\qquad u\ge0.
\end{equation}
Then:
\begin{equation}\label{int_TV}
d_{\mathrm{TV}}\bigl(\mathrm{Ber}(p),\mathrm{Ber}(q)\bigr)
\asymp
\int_0^\infty w(t)
\Bigl(1-\prod_{i=1}^n\bigl(1-\lambda_i+\lambda_ie^{-ta_i^2}\bigr)\Bigr)\,dt.
\end{equation}
Indeed, from
\[
\min\{1,u\}\asymp1-e^{-u},
\]
we get: 
\[
\mathbb E\min\{1,\sqrt G\}\asymp\mathbb E[1-e^{-\sqrt G}].
\]
By \eqref{levy_laplace} and independence,
\[
\mathbb E\bigl[1-e^{-\sqrt G}\bigr]
=
\int_0^\infty w(t)\bigl(1-\mathbb Ee^{-tG}\bigr)\,dt
=
\int_0^\infty w(t)
\Bigl(1-\prod_{i=1}^n\bigl(1-\lambda_i+\lambda_ie^{-ta_i^2}\bigr)\Bigr)\,dt,
\]
and \eqref{int_TV} follows.
\smallskip%

The defining formula for
\(d_{\mathrm{TV}}(\mathrm{Ber}(p),\mathrm{Ber}(q))\)
contains \(2^n\) terms; see Bhattacharyya et al.~\cite{BhattacharyyaEtAl}
for the complexity of exact evaluation.
By contrast, \eqref{int_TV} is a one-dimensional integral of a nonnegative function.
\smallskip%

We now discuss related work.

\smallskip%
Kontorovich~\cite{Kont25} asks for tractable upper and lower bounds on
\(d_{\mathrm{TV}}(\mathrm{Ber}(p),\mathrm{Ber}(q))\) whose ratio is bounded by an absolute constant, uniformly in
\(p,q\) and \(n\). Theorem~\ref{main} gives such bounds.
\smallskip%

Theorem~\ref{main} appeared simultaneously with the independent work of
Avital, Kontorovich, Vershynin, and Zou~\cite{AKVZ}, who obtained an
analytic bound for arbitrary product 
measures. Partial results in this direction include Kontorovich's homogenization bounds~\cite{Kont26} and the regime \(p,q=o(1)\) treated by Avital et al.~\cite{Avital-Kont}.
\smallskip%

On the algorithmic side, Feng et al.~\cite{FengEtAll} gave a
polynomial-time randomized approximation algorithm for TV distance
between product measures. Feng et al.~\cite{FengLiuLiu}
subsequently gave a deterministic 
algorithm. Both have superlinear running time in \(n\). More recently, Anand et al.~\cite{ABG} obtained a randomized linear-time algorithm.
\smallskip%

Our proof is elementary. In \S\,\ref{cnot}, we prove the upper bound. In \S\,\ref{lower_bound}, we prove the matching lower bound. In \S\,\ref{proof_main}, we complete the proof of Theorem~\ref{main}. 
\smallskip%

\textbf{Acknowledgment and AI assistance.}
This paper was developed through extensive interaction with OpenAI's
ChatGPT, which was used to explore and refine proofs, check mathematics, and edit the manuscript.

\section{The upper bound}\label{cnot}

For brevity, write:
\[
P=\mathrm{Ber}(p),
\qquad
Q=\mathrm{Ber}(q).
\]
Let \(P\otimes Q\) denote the law on \(\{0,1\}^n\times\{0,1\}^n\) given by
\[
(P\otimes Q)(x,y)=P(x)Q(y),
\]
that is, the joint law of independent \(X\sim P\), \(Y\sim Q\). Similarly,
\(Q\otimes P\) is the joint law of independent \(X'\sim Q\), \(Y'\sim P\).
Set:
\[
d_{\mathrm{TV}}=d_{\mathrm{TV}}(P,Q),
\qquad
\widehat d=d_{\mathrm{TV}}(P\otimes Q,\,Q\otimes P).
\]
\begin{lemma}\label{doubling}
We have
\[
d_{\mathrm{TV}}\le\widehat d\le2\,d_{\mathrm{TV}}.
\]
\end{lemma}

\begin{proof}
The projection \((x,y)\mapsto x\) maps \(P\otimes Q\) to \(P\) and \(Q\otimes P\) to \(Q\). Since taking 
images cannot increase total
variation, \(d_{\mathrm{TV}}\le\widehat d\). Conversely, by the triangle inequality,
\[
\widehat d
\le
d_{\mathrm{TV}}(P\otimes Q,\,Q\otimes Q)
+
d_{\mathrm{TV}}(Q\otimes Q,\,Q\otimes P)
=
2\,d_{\mathrm{TV}},
\]
and the proof follows. \qed
\end{proof}

For probability measures \(\mu,\nu\) on a finite set, let
\[
\mathrm{BC}(\mu,\nu)=\sum_x\sqrt{\mu(x)\nu(x)}.
\]
Note that \(\mathrm{BC}\) is multiplicative over product measures:
\[
\mathrm{BC}(\mu_1\otimes\mu_2,\nu_1\otimes\nu_2)
=
\mathrm{BC}(\mu_1,\nu_1)\,\mathrm{BC}(\mu_2,\nu_2).
\]

\begin{lemma}\label{lecam}
We have
\[
d_{\mathrm{TV}}(\mu,\nu)\le\sqrt{1-\mathrm{BC}(\mu,\nu)^2}.
\]
\end{lemma}

\begin{proof}
By the Cauchy--Schwarz inequality,
\begin{multline*}
d_{\mathrm{TV}}(\mu,\nu)
=
\frac12\sum_x\bigl|\sqrt{\mu(x)}-\sqrt{\nu(x)}\bigr|
\bigl(\sqrt{\mu(x)}+\sqrt{\nu(x)}\bigr)
\le\\
\le
\frac12
\sqrt{\sum_x\bigl(\sqrt{\mu(x)}-\sqrt{\nu(x)}\bigr)^2}\,
\sqrt{\sum_x\bigl(\sqrt{\mu(x)}+\sqrt{\nu(x)}\bigr)^2}
=
\frac12\sqrt{(2-2\,\mathrm{BC})(2+2\,\mathrm{BC})},
\end{multline*}
and the proof follows. \qed
\end{proof}

We now compute \(\widehat d\) by conditioning on the coordinates where the
two samples disagree. Let \(X\sim P\), \(Y\sim Q\) be independent, let
\(X'\sim Q\), \(Y'\sim P\) be independent, and set:
\[
D_i=X_i\oplus Y_i,
\qquad
D'_i=X'_i\oplus Y'_i,
\qquad
i=1,\ldots,n.
\]
Thus \(D_i=1\) when \(X_i\ne Y_i\). Since
\((x,y)\mapsto(x,x\oplus y)\) is a bijection of
\(\{0,1\}^n\times\{0,1\}^n\),
\begin{equation}\label{hat_d_XD}
\widehat d
=
d_{\mathrm{TV}}\bigl(\mathrm{law}(X,D),\,\mathrm{law}(X',D')\bigr).
\end{equation}
The pairs \((X_i,D_i)\), \(i=1,\ldots,n\), are independent, and so are the
pairs \((X'_i,D'_i)\). Their laws are as follows:
\[
\begin{array}{c|c|c}
(x,d) & \mathbb P(X_i=x,\,D_i=d) & \mathbb P(X'_i=x,\,D'_i=d)\\\hline
(0,0) & (1-p_i)(1-q_i) & (1-p_i)(1-q_i)\\
(1,0) & p_iq_i & p_iq_i\\
(0,1) & (1-p_i)q_i & p_i(1-q_i)\\
(1,1) & p_i(1-q_i) & (1-p_i)q_i
\end{array}
\]
Three facts follow from the table.
\smallskip%

First, \(D\) and \(D'\) have the same law: \(D_i,D'_i\sim
\mathrm{Ber}(\lambda_i)\) independently, where
\[
\lambda_i=p_i(1-q_i)+(1-p_i)q_i.
\]
Second, given \(D_i=0\), the conditional law of \(X_i\) coincides with the
conditional law of \(X'_i\) given \(D'_i=0\).
\smallskip%

Third, given \(D_i=1\), the two conditional laws form a mirror pair. For
\(\lambda_i>0\), set:
\[
\alpha_i=\frac{p_i-q_i}{\lambda_i},
\qquad
|\alpha_i|=a_i.
\]
Since
\[
p_i(1-q_i)+(1-p_i)q_i=\lambda_i,
\qquad
p_i(1-q_i)-(1-p_i)q_i=p_i-q_i,
\]
we have:
\[
\mathbb P(X_i=1\mid D_i=1)
=
\frac{p_i(1-q_i)}{\lambda_i}
=
\frac{1+\alpha_i}2,
\qquad
\mathbb P(X'_i=1\mid D'_i=1)
=
\frac{(1-p_i)q_i}{\lambda_i}
=
\frac{1-\alpha_i}2.
\]
In particular, \(|\alpha_i|=a_i\le1\), since \((1\pm\alpha_i)/2\) are
probabilities.
\smallskip%

Let \(\pi\) be the common law of \(D\) and \(D'\). 
Pick \(d\in\{0,1\}^n\) with \(\pi(d)>0\). 
Let \(\tau_d\) be the TV distance
between the conditional law of \(X\) given \(D=d\) and the conditional law
of \(X'\) given \(D'=d\). Splitting the sum in \eqref{hat_d_XD}:
\begin{multline*}
\widehat d
=
\frac12\sum_d\sum_x
\bigl|\mathbb P(X=x,D=d)-\mathbb P(X'=x,D'=d)\bigr|
= \\ 
= \sum_d\pi(d)\,
\frac12\sum_x
\bigl|\mathbb P(X=x\mid D=d)-\mathbb P(X'=x\mid D'=d)\bigr|,
\end{multline*}
that is,
\begin{equation}\label{average_tv}
\widehat d=\mathbb E\,\tau_D.
\end{equation}
Set:
\[
I(d)=\{i:d_i=1\}.
\]
Both conditional laws are product measures. By the second fact, their
factors with \(i\notin I(d)\) coincide; a common factor does not affect
total variation. By the third fact,
\begin{equation}\label{tau_mirror}
\tau_d
=
d_{\mathrm{TV}}(R_d^+,R_d^-),
\qquad
R_d^\pm
=
\bigotimes_{i\in I(d)}\mathrm{Ber}\Bigl(\frac{1\pm\alpha_i}2\Bigr).
\end{equation}
For a single mirror pair,
\[
\mathrm{BC}\Bigl(
\mathrm{Ber}\Bigl(\frac{1+\alpha_i}2\Bigr),
\mathrm{Ber}\Bigl(\frac{1-\alpha_i}2\Bigr)
\Bigr)
=
2\sqrt{\frac{1+\alpha_i}2\cdot\frac{1-\alpha_i}2}
=
\sqrt{1-a_i^2}.
\]
By multiplicativity and Lemma~\ref{lecam},
\begin{equation}\label{Delta_d}
\tau_d
\le
\Delta_d
:=
\sqrt{1-\prod_{i=1}^n(1-a_i^2)^{d_i}}.
\end{equation}
Since 
\[
1-\prod(1-x_i)\le\min\{1,\sum x_i\}\quad 
\text{for \(x_i\in[0,1]\)},
\]
it follows that:
\[
\Delta_d
\le
\min\Bigl\{1,\sqrt{\sum_{i=1}^na_i^2d_i}\Bigr\}.
\]
Averaging over \(D\),
\[
\mathbb E\,\Delta_D \le \mathbb E\min\{1,\sqrt G\}.
\]
From \eqref{average_tv} and Lemma~\ref{doubling}:
\begin{equation}\label{upper_chain}
d_{\mathrm{TV}}
\le
\widehat d
=
\mathbb E\,\tau_D
\le
\mathbb E\,\Delta_D
\le
\mathbb E\min\{1,\sqrt G\},
\end{equation}
and the upper bound follows. \qed

\section{The lower bound}\label{lower_bound}

By \eqref{tau_mirror}, \(\tau_d\) is the TV distance between two products of mirror pairs. We prove the following lower bound, matching \eqref{Delta_d}.
\begin{theorem}\label{dephasing_thm}
Let \(\alpha_1,\ldots,\alpha_m\in[-1,1]\), and set:
\[
R^\pm
=
\bigotimes_{i=1}^m\mathrm{Ber}\Bigl(\frac{1\pm\alpha_i}2\Bigr).
\]
Then:
\[
d_{\mathrm{TV}}(R^+,R^-)
\ge
c\min\left\{
1,
\sqrt{\sum_{i=1}^m\alpha_i^2}
\right\}
\]
for an absolute constant \(c>0\).
\end{theorem}

\begin{proof}
Set:
\[
\sigma^2=\sum_{i=1}^m\alpha_i^2.
\]
Assume \(\sigma>0\); otherwise, there is nothing to prove.
\smallskip%

For \(x\in\{0,1\}^m\), set:
\[
z_i=2x_i-1\in\{-1,1\}.
\]
For \(t>0\), consider the test function:
\[
f_t(x)=\exp\Bigl\{it\sum_{i=1}^m\alpha_iz_i\Bigr\}.
\]
Since \(|f_t|=1\),
\[
d_{\mathrm{TV}}(R^+,R^-)
\ge
\frac12
\left|
\sum_xf_t(x)R^+(x)
-
\sum_xf_t(x)R^-(x)
\right|.
\]
Under \(R^+\), the coordinates are independent, and \(z_i=1\) with
probability \((1+\alpha_i)/2\). Hence:
\[
\sum_xf_t(x)R^+(x)
=
\prod_{i=1}^m
\Bigl(
\frac{1+\alpha_i}2\,e^{it\alpha_i}
+
\frac{1-\alpha_i}2\,e^{-it\alpha_i}
\Bigr)
=
\prod_{i=1}^m
\bigl(\cos(t\alpha_i)+i\alpha_i\sin(t\alpha_i)\bigr)
=:
\zeta.
\]
Under \(R^-\), the test function \(f_t\) is unchanged, but \(z_i=1\)
with probability \((1-\alpha_i)/2\). Hence, the \(i\)-th factor becomes:
\[
\frac{1-\alpha_i}2\,e^{it\alpha_i}
+
\frac{1+\alpha_i}2\,e^{-it\alpha_i}
=
\cos(t\alpha_i)-i\alpha_i\sin(t\alpha_i),
\]
and the corresponding sum is \(\bar\zeta\). Consequently,
\begin{equation}\label{imag_lower}
d_{\mathrm{TV}}(R^+,R^-)
\ge
|\Im\zeta|.
\end{equation}
Choose a small absolute \(\varepsilon>0\), and set:
\[
t
=
\varepsilon
\min\left\{
\frac1{\sigma^2},\frac1\sigma
\right\}.
\]
Since \(|\alpha_i|\le\min\{1,\sigma\}\),
\[
|t\alpha_i|\le\varepsilon,
\qquad
t\sigma^2=\varepsilon\min\{1,\sigma\},
\qquad
t^2\sigma^2\le\varepsilon^2.
\]
Taylor expansion gives:
\[
\log\bigl(\cos(t\alpha_i)+i\alpha_i\sin(t\alpha_i)\bigr)
=
it\alpha_i^2
+
O(t^2\alpha_i^2)
+
iO(t^3\alpha_i^4),
\]
where the \(O(t^2\alpha_i^2)\)-term is real. Summing
over \(i\), and using 
\[
\sum_i\alpha_i^4\le\sigma^4,
\]
we get:
\[
\log\zeta
=
it\sigma^2
+
O(t^2\sigma^2)
+
iO(t^3\sigma^4).
\]
Since \(t^2\sigma^2\le\varepsilon^2\),
\[
\log\zeta
=
O(\varepsilon^2)
+
it\sigma^2\bigl(1+O(\varepsilon^2)\bigr).
\]
Hence:
\[
|\zeta|
=
e^{O(\varepsilon^2)}
=
1+O(\varepsilon^2),
\qquad
\arg\zeta
=
t\sigma^2\bigl(1+O(\varepsilon^2)\bigr).
\]
Since \(t\sigma^2\le\varepsilon\),
\[
\sin(\arg\zeta)
=
t\sigma^2\bigl(1+O(\varepsilon^2)\bigr).
\]
Consequently,
\[
\Im\zeta
=
|\zeta|\sin(\arg\zeta)
=
t\sigma^2\bigl(1+O(\varepsilon^2)\bigr).
\]
Choose \(\varepsilon>0\) so small that the last factor is at least
\(1/2\). Then:
\[
\Im\zeta
\ge
\frac12\,t\sigma^2
=
\frac\varepsilon2\min\{1,\sigma\},
\]
and the proof follows. \qed
\end{proof}

\section{Proof of Theorem~\ref{main}}\label{proof_main}

The upper bound is \eqref{upper_chain}. For the lower bound, let
\(d\in\{0,1\}^n\) with \(\pi(d)>0\). Set:
\[
G_d=\sum_{i=1}^na_i^2d_i.
\]
Applying Theorem~\ref{dephasing_thm} to
\((\alpha_i)_{i\in I(d)}\),
\[
\tau_d\ge c\min\{1,\sqrt{G_d}\}.
\]
Averaging over \(D\), 
\[
\mathbb E\,\tau_D \ge c\, \mathbb E \min\{1, \sqrt{G}\}.
\]
From Lemma~\ref{doubling}:
\[
d_{\mathrm{TV}}
\ge
\frac12\,\widehat d
=
\frac12\,\mathbb E\,\tau_D,
\]
and the proof follows. \qed

\bibliographystyle{plain}
\bibliography{ref}

\end{document}